\documentclass[a4paper,12pt]{article}
\usepackage{amsmath}
\usepackage{amsthm}
\usepackage{amssymb}
\usepackage{amscd}
\usepackage{graphicx}
\usepackage{epsfig}
\usepackage[matrix,arrow,curve]{xy}
\usepackage{color}
\usepackage{mathrsfs}
\usepackage[scr=rsfs,cal=boondox]{mathalpha}

\usepackage{floatrow}
\usepackage{graphics}
\usepackage{color}

\usepackage{stmaryrd}
\usepackage{hyperref}
\usepackage{multirow}

\usepackage{bbm}
\usepackage{latexsym}
\usepackage{amsfonts}
\input xy
\usepackage{tikz}
\usetikzlibrary{matrix}
\usetikzlibrary{decorations.pathreplacing,angles,quotes}

\newtheorem{theorem}{Theorem}[section]

\theoremstyle{definition}
\newtheorem{example}[theorem]{Example}
\newtheorem{remark}[theorem]{Remark}
\newtheorem{definition}[theorem]{Definition}

\theoremstyle{definition}

\mathchardef\za="710B  
\mathchardef\zb="710C  
\mathchardef\zg="710D  
\mathchardef\zd="710E  
\mathchardef\zve="710F 
\mathchardef\zz="7110  
\mathchardef\zh="7111  
\mathchardef\zvy="7112 
\mathchardef\zi="7113  
\mathchardef\zk="7114  
\mathchardef\zl="7115  
\mathchardef\zm="7116  
\mathchardef\zn="7117  
\mathchardef\zx="7118  
\mathchardef\zp="7119  
\mathchardef\zr="711A  
\mathchardef\zs="711B  
\mathchardef\zt="711C  
\mathchardef\zu="711D  
\mathchardef\zvf="711E 
\mathchardef\zq="711F  
\mathchardef\zc="7120  
\mathchardef\zw="7121  
\mathchardef\ze="7122  
\mathchardef\zy="7123  
\mathchardef\zvw="7124  
\mathchardef\zvr="7125 
\mathchardef\zvs="7126 
\mathchardef\zf="7127  
\mathchardef\zG="7000  
\mathchardef\zD="7001  
\mathchardef\zY="7002  
\mathchardef\zL="7003  
\mathchardef\zX="7004  
\mathchardef\zP="7005  
\mathchardef\zS="7006  
\mathchardef\zU="7007  
\mathchardef\zF="7008  
\mathchardef\zW="700A  
\mathchardef\zC=\Psi

\newcommand{\be}{\begin{equation}}
\newcommand{\ee}{\end{equation}}

\newcommand{\bea}{\begin{eqnarray}}
\newcommand{\eea}{\end{eqnarray}}
\newcommand{\beas}{\begin{eqnarray*}}
\newcommand{\eeas}{\end{eqnarray*}}
\newcommand{\we}{\wedge}

\newcommand{\ot}{\otimes}

\newcommand{\g}{\mathfrak{g}}

\newcommand{\Ci}{C^{\infty}}

\newcommand{\pa}{\partial}
\newcommand{\ti}{\times}

\newcommand{\ad}{{\mathrm{ad}}}

\newcommand{\cN}{{\mathcal N}}

\newcommand{\cK}{{\mathcal K}}
\newcommand{\cT}{{\mathcal T}}

\newcommand{\cX}{\mathcal{X}}

\newcommand{\cZ}{\mathcal{Z}}

\newcommand{\Sec}{\operatorname{Sec}}
\newcommand{\Ad}{\operatorname{Ad}}

\def\br{{\mathbf r}}

\def\sJ{{\mathsf J}}

\def\sT{{\mathsf T}}
\def\sV{{\mathsf V}}

\def\sj{{\mathsf j}}

\def\xi{\mathrm{i}}

\def\dt{\xd_{\sT}}

\def\cN{{\mathcal N}}

\def\b|{\,\big|\,}
\newdir{|>}{%
!/4.5pt/@{|}*:(1,-.2)@^{>}*:(1,+.2)@_{>}}

\def\dim{\operatorname{dim}}

\newdir{ (}{{}*!/-5pt/@^{(}}

\newcommand{\id}{\mathrm{id}}

\def\g{\mathfrak{g}}
\def\k{\mathfrak{k}}

\newcommand{\half}{{\frac{1}{2}}}

\newcommand{\R}{{\mathbb R}}

\newcommand{\C}{{\mathbb C}}
\newcommand{\Z}{{\mathbb Z}}

\newcommand{\xd}{\textnormal{d}}

\newcommand{\mn}{{\medskip\noindent}}

\newcommand{\no}{{\noindent}}

\newdir{|>}{%
!/4.5pt/@{|}*:(1,-.2)@^{>}*:(1,+.2)@_{>}}
\newdir{ (}{{}*!/-5pt/@^{(}}

\newfloat{pict}{h}{lob}
\floatname{pict}{Fig.}

\def\bl{\big( }
\def\br{\big) }
\def\Bl{\Big( }
\def\Br{\Big) }

\def\z2{{\Z_2^n}}

\def\bl{\big(}
\def\br{\big)}
\def\Bl{\Big(}
\def\Br{\Big)}

\begin{document}
\title{\textbf{Nijenhuis operators on Banach fibrations}\footnote{This research was partially funded by the National Science Centre (Poland) within the project WEAVE-UNISONO, No. 2023/05/Y/ST1/00043.}}
\author{Katarzyna Grabowska\footnote{email:konieczn@fuw.edu.pl }\\
\textit{Faculty of Physics,
                University of Warsaw}
\\ \\
Janusz Grabowski\footnote{email: jagrab@impan.pl} \\
\textit{Institute of Mathematics, Polish Academy of Sciences}}
\date{}
\maketitle
\begin{abstract}
In the infinite-dimensional Banach setting, we consider general smooth Banach fibrations $\zt:M\to M_0$ and `$(1,1)$-tensors' $N:\sT M\to\sT M$ that are projectable (in the obvious sense) onto Nijenhuis operators $N_0:\sT M_0\to\sT M_0$ on $M_0$. We prove that the vanishing of the Nijenhuis torsion of $N_0$ is equivalent to the fact that the Nijenhuis torsion of $N$ takes only vertical values, i.e., values in $\ker(\sT\zt)$. Consequences for almost complex structures on (real) Banach manifolds are also derived. As canonical examples, we define tangent lifts $\dt(N_0):\sT\sT M_0\to\sT\sT M_0$ of Nijenhuis operators $N_0$ in the Banach category, and prove that they are automatically projectable for the canonical fibrations $\zt_{M_0}:\sT M_0\to M_0$. Finally, we comment on the projectability in the case of Banach homogeneous manifolds $\zt:G\to G/K$, studied recently by some authors.

\par


\
\end{abstract}
\section{Introduction}
Many aspects of differential geometry in the infinite-dimensional Banach setting have not been systematically studied yet. One of them is the concept of a \emph{Nijenhuis operator}, a powerful tool in the theory of integrable systems \cite{Magri:1978,Magri:1984}. This theory in the infinite-dimensional Banach setting is substantially more complicated, as many standard methods and proofs do not work. An infinite-dimensional version of the Newlander-Nirenberg Theorem was originated in \cite{Beltita:2005}, and Nijenhuis operators on Banach homogeneous spaces $G/K$ were recently studied in \cite{Golinski:2025}, which was an inspiration for our work.

The authors study in \cite{Golinski:2025} the following question: what $G$-invariant Nijenhuis operators $N_0$ on a Banach homogeneous space $M_0=G/K$ come from invariant operators on the (real) Banach Lie group $G$ \emph{via} the canonical projection $\zt:G\to G/K$? It is easy to see that, due to the homogeneity, everything reduces to the Lie algebra $\g$ of $G$ and to the study of the Nijenhuis torsion in the Lie algebra sense.

We realized that this question has a more general geometric context, and proved an analogous result, skipping the homogeneity assumption, and working with general fibrations $\zt:M\to M_0$ replacing the principal bundle $\zt: G\to G/K$.

\mn Let us recall that a Nijenhuis operator on a Lie algebra $\g$ is a linear map $N:\g\to\g$ whose Nijenhuis torsion,
$$\cT_N(X,Y)=[NX,NY]-N\big([NX,Y]+[X,NY]-N[X,Y]\big)$$
vanishes. This concept was introduced by Nijenhuis in \cite{Nijenhuis:1951}, and the condition $\cT_N=0$ easily implies that the `contracted' bracket,
\be\label{contr}[X,Y]_N=[NX,Y]+[X,NY]-N[X,Y],\ee
is again a Lie bracket.

\mn In differential geometry, one considers \emph{Nijenhuis tensors} as $(1,1)$ tensors $N$ on a manifold $M$, representing a vector bundle morphism $N:\sT M\to \sT M$ and the Nijenhuis torsion is defined in terms of the Lie bracket of vector fields. A fundamental observation in this context is the fact that the Nijenhuis torsion, defined originally on vector fields, is, in fact, represented by a $(1,2)$-tensor,
$$\cT_N:\we^2\sT M\to\sT M.$$
This framework can be extended to \emph{Lie algebroids}, where we replace the bracket of vector fields with a bracket on sections of a vector bundle. Moreover, we can consider the compatibility of Nijenhuis tensors with other geometric structures, to mention Poisson-Nijenhuis structures \cite{Grabowski:1997a,Kosmann:1990}. Note that Nijenhuis structures for more general algebraic structures on vector bundles have been discussed in \cite{Carinena:2001}.

\mn Especially important is the case of \emph{almost complex structures}, i.e., $(1,1)$ tensors $N:\sT M\to\sT M$ satisfying $N^2=-\id$. The celebrated Newlander-Nirenberg Theorem \cite{Newlander:1957} says that, in this case, $\cT_N$ vanishes exactly when the almost complex structure arises from a genuine complex structure on $M$.

This theory in the infinite-dimensional Banach setting looks substantially more complicated, as many standard methods do not work, and global vector fields on a given Banach manifold may not exist. The starting point is proving that $\cT_N$ is, actually, tensorial, i.e., it is a vector bundle map
$$\cT_N:\sT M\we_M\sT M\to\sT M.$$
This proof is nontrivial in infinite dimensions.

\mn Following \cite{Golinski:2025}, we consider a Banach manifold $M$ and \emph{$\cN$-operators} $N:\sT M\to\sT M$, but we generalize the observations done for fiber bundles $\zt:G\to G/K$ to the case of arbitrary Banach fibrations $\zt:M\to M_0$. In particular, we define \emph{projectable $\cN$-operators} in this setting, and show that if $N$ is projectable onto $N_0:\sT M_0\to\sT M_0$, then the projected operator $N_0$ is Nijenhuis if and only if $\cT_N$ is \emph{vertical}, i.e., it takes only vertical values,
$$\cT_N:\sT M\we_M\sT M\to \sV,$$
where $\sV=\ker(\sT\zt)\subset\sT M$ is the vertical subbundle. The particular case of almost complex structures, $N^2=-\id$, is also studied.

\mn Entirely novel observation is that the tangent (complete) lifts produce projectable Nijenhuis operators $\dt(N)$ on $\zt_M:\sT M\to M$ from Nijenhuis ones $N$ on $M$. In finite dimensions, a more general version says  \cite{Grabowska:2022} that the complete lifts of vector-valued forms on $M$ to the higher tangent bundles $\sT^kM$ respect the Fr\"olicher-Nijenhuis bracket $[\cdot,\cdot]_{FN}$. Note that for $\cN$-operators (vector valued 1-forms) we have $\cT_N=\half[N,N]_{FN}$. We do not have an analogous result for infinite-dimensional vector spaces. Still, we define the tangent lift $\dt$ of $\cN$-operators in infinite dimensions and show that $\dt(N)$ is Nijenhuis and projectable for $N$ being Nijenhuis.

\mn Finally, we show how to derive the main results of \cite{Golinski:2025} in the homogeneous case from our general setting.

\section{Vector bundles in the Banach category}
We will work in the category of smooth real Banach manifolds \cite{Dodson:2016,Hamilton:1982,Lang:1999}, using the standard \emph{Bastiani differentiability} on Banach spaces:

\mn Let $E,F$ be (real) Banach spaces and \( f : E \to F \) be a smooth ($\Ci$) map.
    \begin{itemize}
      \item The derivative \( D_pf \in {B}(E, F) \) at $p\in E$ is a bounded linear map.
      \item Higher derivatives \( D^k_p f \in {B}^k(E; F) \) are bounded $k$-linear maps.
    \end{itemize}

\begin{definition} \emph{Banach vector bundles} $\zt:V\to M$ are locally trivial fibrations with a typical fiber being a Banach space $F$ such that the transition maps take values in $B(F)$. \emph{Morphisms} of vector bundles (VB-morphisms) are smooth fibration morphisms which are bounded linear maps on fibers.
\end{definition}
\no Locally, we can view $V$ as $E\ti F$, so smooth sections $X$ of $\zt$ can be locally viewed as smooth maps $X:E\to F$. The space of smooth sections $\Sec(V)$ is canonically a module over the algebra $\Ci(M)$ of smooth functions on $M$. Vector bundle morphisms are locally represented by smooth maps
$$\zF:E\ti F\to E'\ti F', \quad \zF(p,y)=\bl\zF_0(p),\zF_p(y)\br,$$
where $\zF_0:E\to E'$ is a smooth map and $\zF_p\in B(F,F')$.

\mn A canonical example is the tangent bundle $\zt_M:\sT M\to M$, whose space $\cX(M)$ of sections consists of vector fields. If $M$ is modelled on a Banach space $E$, then, locally, $\sT M=E\ti E$, and vector fields are just smooth maps $X:E\to E$. The derivative $D_xf$ of a smooth map $f:M\to M'$ between manifolds  modeled on Banach spaces $E$ and $F$ is now interpreted as a map
$$D_xf:\sT_xM\to\sT_{f(x)}M'.$$
The definition is consistent, since local derivatives $D_xf:E\ti E\to F\ti F$ transform properly with respect to transition maps.

\mn Note, however, that not all Banach manifolds, even second countable, admit smooth `bump functions', so the existence of global sections of a vector bundle is not guaranteed. For instance, global smooth functions or global smooth vector fields could be only trivial ones. Therefore, in our definitions of geometrical objects and the corresponding considerations, we will work exclusively locally, which will be fully sufficient for our purposes. Therefore, in our proofs, we can assume that our manifolds are just (real) Banach spaces.

Note that on the space $\cX(M)$ of smooth vector fields, a canonical Lie algebra bracket. If $M$ is modelled on a Banach space $E$, then, locally, then $\sT M=E\ti E$, and vector fields are just smooth maps $X:E\to E$, and the bracket of vector fields reads
\be\label{br}[X,Y]_p=D_pX(Y_p)-D_pY(X_p).\ee
It is also easy to see that
\be\label{Liea}
[X,fY]=f[X,Y]+X(f)Y,
\ee
for any function $f$ on $M$, which is usually expressed as the fact that
$$\ad_X:\cX(M)\to\cX(M),\quad \ad_X(Y)=[X,Y]$$
is a VB-derivation, and that $\sT M$ is a \emph{Lie algebroid}.

\mn Starting with a vector bundle, we can canonically construct other vector bundles, namely \emph{jet bundles} of their (local) sections. It reduces to the case of trivial vector bundles $V=E\ti F$, for which the \emph{$k$-th jet bundle} is
$$
\sJ^k(E) = E \times\bl F \times {B}(E, F) \times \cdots \times {B}^k(E, F)\br.
$$
The fibers of these jet bundles are naturally Banach manifolds, so fibers of the infinite jet bundle have a natural Fr\'echet topology. Any $X\in\Sec(V)$ induces canonically a smooth section $\sj^k(X)$ of $\sJ^k(V)$, by
$$\sj^k_p(X)=\bl X_p,D_pX,\cdots,D^k_pX\br.$$
On the vector space $\Sec(V)$ we have a natural topology, called the \emph{Whitney $\Ci$-topology} (cf. \cite{Dodson:2016,Hamilton:1982}). In the literature, one can find other natural topologies, e.g., point-wise $C^\infty$-convergence, or the $c^\infty$-topology, playing a fundamental role in the so-called \emph{convenient setting} for differential calculus \cite{Kriegl:1997}.

\mn Now, we can easily define linear differential operators between sections of two Banach vector bundles over the same base $M$.
\begin{definition}
An $\R$-linear map $\zf:\Sec(V)\to\Sec(V')$ we call a \emph{linear differential operator of degree $\le k$} if $\zf$ comes from a VB-morphism $\zF:\sJ^k(V)\to V'$ by
$$\zf(X)_p=\zF\bl\sj^k_p(X)\br.$$
In particular, $\zf$ is of degree 0 if and only if it comes from a VB-morphism $\zF:V\to V'$ covering the identity on $M$.
\end{definition}
\no As already mentioned, this definition makes sense when we use local sections only, and the above definition is correct, as it is invariant with respect to the transition maps. In particular, the adjoint map associated with the Lie bracket of vector fields, $\ad_X(Y)=[X,Y]$, is a first-order differential operator.

\begin{remark} Any Banach VB-morphism $\zF:V\to V$ covering the identity on $M$ induces a $\Ci(M)$-linear map $\zf:\Sec(V)\to\Sec(V)$ by
\be\label{sct} \zf(X)_x=\zF(X_x).\ee
For any function $f$ on $M$, we have $\zf(fX)=f\zf(X)$. If $\dim(F)=n<\infty$, the converse is also valid: if $\zf:\Sec(V)\to\Sec(V)$ is $\Ci(M)$-linear, then $\zf$ is of the form (\ref{sct}) for a uniquely defined VB-morphism $\zF$. This is because if $\zf$ is $\Ci(M)$-linear, and $(e_i)$ is a basis of $F$, then in a trivialization $E\ti F$ we have
\be\label{vbm}\Phi=\ze^i\ot\zf(e_i),\ee
where $(\ze^i)$ is the dual basis of $(e_i)$.

Of course, the construction (\ref{vbm}) of $\zF$ from $\zf$ fails in infinite dimensions, where a continuous $C^\infty(M)$-linear operator $\zf:\Sec(V)\to\Sec(V)$ need not to be a tensor. This means that we cannot use just the module structure on $\Sec(V)$ to check tensoriality. At the beginning, the problem starts from a possible lack of nontrivial global sections, i.e., the fact that $\Sec(V)$ could be $\{0\}$, caused by the lack of partitions of unity. We solve this problem by working with local sections. For many interesting vector bundles, e.g., tangent bundles of Banach Lie groups, global sections span every fiber, so one can work exclusively with such vector bundles. Another possible approach is \emph{via} sheaves of local sections of $V$, and VB-morphisms replaced with sheaves of VB-morphisms. We will not discuss such aspects here and work with the general case, not insisting on the existence of global sections. Note, however, that even locally, working with the $\Ci(M)$-module structure on the space of sections is not enough.
\end{remark}
\begin{example}\label{Vin}
Recall that, if $M=E$ is just a Banach space, the tangent bundle reads $\sT E=E\ti E$, and any vector field $X\in\cX(E)$ is represented by a smooth map $X:E\to E$.

\mn Let us assume that the Banach space $E$ is infinite-dimensional. In the Banach algebra $B(E)$, there is a unique nontrivial closed ideal, $\cK(E)$, consisting of compact operators. Let $\zm\in B(E)^*$ be a non-zero functional on $B(E)$ vanishing on $\cK(E)$ and $Y$ be a fixed vector field on $E$. Consider now $\zf:\cX(E)\to\cX(E)$ given by
$$\zf(X)(p)=\zm\bl D_pX\br\cdot Y_p\,.$$
It is clear that $\zf:\cX(E)\to\cX(E)$ is continuous in any reasonable $\Ci$-topology and it does not come from a tensor $\zF:\sT E\to\sT E$, since it makes use of the derivative of $X$. In other words, $\zf$ is a first-order linear differential operator represented by a vector bundle morphism
$$\zF:\sJ^1(\sT E)\to\sT E.$$
What is surprising is the fact that, nevertheless, $\zf$ is $C^\infty(E)$-linear. Indeed, let $f$ be a smooth function on $E$. Since $D(fX)=fD(X)+X\ot\xd f$, for each $p\in E$, we have
$$\bl \zf(fX)-f\zf(X)\br(p)=\zm\bl X_p\ot\xd f(p)\br\cdot Y_p=0,$$
as $X_p\ot\xd f(p)$ is a rank 1, thus a compact operator.

\end{example}

\begin{remark} In finite dimensions, any $\R$-linear operator $\zf:\Sec(V)\to\Sec(V')$ between sections of vector bundles $V$ and $V'$ over the same manifold $M$, which commutes with the multiplication by functions on $M$, $[\zf,f](Y)=\zf(fY)-f\zf(Y)=0$ for any function $f\in\Ci(M)$ and any section $Y$ of $E$, is automatically a VB-morphism (cf. (\ref{vbm})). Inductively: if $[\zf,f]$ is a linear differential operator of degree $\le k$ for all $f\in\Ci(M)$, then $\zf$ is a linear differential operator of degree $\le k+1$ \cite{Vinogradov:1975}. This is the starting point of the powerful Vinogradov's algebraic approach to differential geometry \cite{Nestruev:2003}. Example \ref{Vin} shows that this approach fails for infinite-dimensional Banach manifolds.
\end{remark}

\section{$\cN$-operators on Banach manifolds}
Let $M$ be a Banach manifold modeled on a Banach space $E$, let $\zt_M:\sT M\to M$ be its tangent bundle, and $\cX(M)$ be the Lie algebra of vector fields on $M$. Here, we regard $\sT M$ as consisting of \emph{velocity vectors} \cite{Dodson:2016}, and we identify locally vector fields on $M$ as smooth maps $X:E\to E$.
\begin{definition} An \emph{$\cN$-operator} on $M$ is a vector bundle morphism
$$N:\sT M\to\sT M$$
covering the identity on $M$. The \emph{Nijenhuis torsion} of $N$ in this setting is an $\R$-bilinear anti-symmetric map
$$\cT_N:\cX(M)\ti\cX(M)\to\cX(M),$$
defined by
\be\label{Nt}\cT_N(X,Y)=[NX,NY]-N\big([NX,Y]+[X,NY]-N[X,Y]\big).\ee
\end{definition}
This definition works well locally, so we can use local vector fields $X,Y$ if we do not have global ones at our disposal. In finite dimension, $\cN$-operators are usually identified with $(1,1)$-tensor fields.
It is known \cite[Theorem 3.2]{Golinski:2025} that the Nijenhuis torsion of $N$ at $p\in M$ depends only on the values of the vector fields at the point $p$, i.e.,
$$\big[\cT_N(X,Y)\big]_p=(\cT_N)_p(X_p,Y_p)$$
for some bilinear and skew-symmetric bounded operator
$$(\cT_N)_p:\sT_pM\ti\sT_pM\to \sT_pM.$$
In other words, $\cT_N$ is a tensor. Working locally, we can put $M=E$ and prove that $\big[\cT_N(X,Y)\big]_p=0$ if $X_p=0$. 

\mn It is also well known that vanishing of $\cT_N$ implies that the `contracted' bracket of vector fields (\ref{contr}) is a \emph{Lie algebroid} bracket (cf. \cite{Anastasiei:2011}) on the vector bundle $\sT M$, with the anchor $N:\sT M\to\sT M$.

\mn Since $\cT_N$ is a tensor, we can easily find its explicit form (cf. \cite[Theorem 3.2]{Golinski:2025}),
\be\label{TN} \cT_N(X,Y)=DN(X,NY)-DN(Y,NX)-N\bl DN(X,Y)-DN(Y,X)\br,
\ee
assuming that $X,Y$ are constant ($DX=DY=0$), that simplifies calculations. Here, we interpret locally:
\begin{align*}
& N:E\to B(E;E)=B(E);\\
& DN:E\to B\bl E;B(E)\br=B^2(E;E).
\end{align*}
$\cN$-operators $N$ with vanishing torsion, $\cT_N=0$, are called \emph{Nijenhuis operators}.

\mn Recall that \emph{almost complex structures} are understood as $\cN$-operators $N$ satisfying $N^2=-\id$. Similarly, $\cN$-operators satisfying $N^2=\id$ are called \emph{almost product structures}, and those satisfying $N^2=0$ -- \emph{almost tangent structures}. If $N$ is additionally Nijenhuis, we deal (cf. \cite{Grabowski:2006}) with \emph{complex structures},    \emph{product structures}, and \emph{tangent structures}, respectively.

\mn  For an $\cN$-operator $N$ on $M$, on the complexified tangent bundle $\sT^\C M$ we have
\be\label{complex}[X+iNX,Y+iNY]=[X,Y]-[NX,NY]+i\bl [NX,Y]+[X,NY]\br,
\ee
so the vector subbundle $Z_+(N)$ in $\sT^C M$ spanned by vectors $X+iNX$ is involutive if and only if
\be\label{7}N\bl[X,Y]-[NX,NY]\br=[NX,Y]+[X,NY].\ee
Replacing $N$ with $-N$ we get a similar condition for $Z_+(-N)=Z_-(N)$. If, additionally, $N$ is almost complex, $N^2=-\id$, (\ref{complex}) is equivalent to
$$\cT_N(X,Y)=[NX,NY]-N\bl[NX,Y]+[X,NY]\br-N^2[X,Y]=0.$$
In other words, exactly like in finite dimensions, an almost complex structure $N$ is complex, $\cT_N=0$, if and only if the complex distribution $Z_+\subset\sT^\C M$ (equivalently, $Z_-$) is involutive.

Note that $Z_+(N)$ can be viewed as the kernel of the complex $\cN$-operator $J_N-i\cdot\id$, where
$$J_N=N^\C:\sT^\C M\to\sT^\C M,\quad J_N(X+iY)=NX+iNY.$$

\section{The case of Banach fibrations}
Let now $\zt:M\to M_0$ be a \emph{fibration}, i.e., $\zt$ is a smooth surjective submersion of Banach manifolds. We are interested in a condition assuring that an $\cN$-operator $N$ on $M$ induces canonically an $\cN$ operator $N_0$ on $M_0$. In other words, the operator $N_0$ is the one making the following diagram commutative,
$$\xymatrix@C+30pt@R+20pt{
\sT M\ar[r]^{N}\ar[d]^{\sT\zt}&\sT M\ar[d]^{\sT\zt}\\
\sT M_0\ar[r]^{N_0}&\sT M_0\,,}
$$
i.e.,
$$
N_0\circ\sT\zt=\sT\zt\circ N.
$$
More precisely,
\be\label{Nproj}N_0([Y_m])=[N(Y_m)],
\ee
where, for simplicity, we write $[X_m]$ for $\sT_m\zt(X_m)$.

Such $\cN$-operators on the total space of the fibration we will call \emph{projectable}. In particular, $N$ maps the \emph{vertical subbundle} $\sV=\ker(\sT\zt)$ of $\sT M$ into $\sV$. We do not assume any complementary conditions for $\sV$: $\sV_m$ may not admit any complementary Banach subspace in $\sT_mM$. Differentiating condition (\ref{Nproj}), we get
\be\label{Nproj1}
DN_0\bl[X_m],[Y_m]\br=[DN(X_m,Y_m)].
\ee
Hence, according to (\ref{TN}), (\ref{Nproj}), and (\ref{Nproj1}),
\beas &&\cT_{N_0}\bl[X_m],[Y_m]\br= \\
&&= DN_0\bl[X_m],N_0[Y_m]\br-DN_0\bl[Y_m],N_0[X_m]\br
-N_0\Bl DN_0\bl[X_m],[Y_m]\br-DN_0\bl[Y_m],[X_m]\br\Br\\
&&=\sT\zt\Bl DN\bl X_m,N(Y_m)\br-DN\bl Y_m,N(X_m)\br-N\bl DN(X_m,Y_m)-DN(Y_m,X_m)\br\Br\\
&&=\sT\zt\bl \cT_N(X_m,Y_m)\br.
\eeas
Since $\sT\zt$ maps $\sT M$ onto $\sT M_0$, the above identity implies our first result.
\begin{theorem}\label{main}
If $N$ is an $\cN$-operator on a fibration $\zt:M\to M_0$ which is projectable onto an $\cN$-operator $N_0$ on $M_0$, then $N_0$ is Nijenhuis if and only if $\cT_N$ is vertical, i.e., $\sT\zt\circ \cT_N=0$.
\end{theorem}
\mn If complex structures are concerned, a projectable $\cN$-operator $N$ on a Banach fibration $\zt:M\to M_0$ projects onto a complex structure $N_0$ if and only if $N^2+\id$ and $\cT_N$ take only vertical values.
\begin{theorem}
A projectable $\cN$-operator $N$ on a Banach fibration $\zt:M\to M_0$ projects onto a complex structure $N_0$ if and only if $N^2+\id$ take only vertical values, and the complex distribution $\hat Z_+\subset\sT^\C M$,
(equivalently, $\hat Z_-$) is involutive. Here,
$$\hat Z_\pm=\{X+iY\in\sT^\C M\,\big|\, Y\mp NX\in V\}.$$
\end{theorem}
\begin{proof}
It is easy to see that $\hat Z_\pm$ are the inverse images of
$$Z_\pm=\{X+iN_0X\in\sT^\C M_0\,\big|\, X\in\sT M_0\}$$ under the projection
$$(\sT\zt)^\C:\sT^\C M\to\sT^\C M_0.$$
Any local complex vector field $X_0+iN_0X_0$ on $M_0$ is the image of a local projectable vector field $X+iNX$ on $M$. As $X+iNX$ is a local section of $\hat Z_+$, and the Lie bracket of projectable vector fields projects onto the Lie bracket of projections. The involutivity of $\hat Z_\pm$ implies the involutivity of $Z_\pm$, so $N_0$ is complex.

Conversely, suppose $N_0$ is a complex structure. Then, according to Theorem \ref{main}, $N^2+\id$ and $\cT_N$ take only vertical values. Let $X+iX_1$ and $Y+iY_1$ be local sections of $\hat Z_+$, so $Z_1=X_1-X,Z_2=Y_1-Y$ are vertical. As $N$ is projectable, $N$ respects the vertical subbundle $\sV$, $N(\sV)\subset\sV$, and we get from (\ref{Nt}) for $X$ being vertical,
$$[X,NY]-N[X,Y]\in \sV$$
for any vector field $Y$ on $M$. This easily implies that
$$[X,Y\pm iNY]=[X,Y]\pm i[X,NY]=[X,Y]\pm iN[X,Y]$$
modulo $\sV$, i.e.,
\be\label{h1}
[V,\hat Z_\pm]\subset\hat Z_\pm.
\ee
Of course, $[\sV,\sV]\subset\sV$. Now, taking into account that local sections of $\hat Z_+$ are of the form $X+iNX$ modulo $\sV$, and doing calculations modulo $\sV$, we easily prove that (\ref{7}) is satisfied modulo $V$, so $\hat Z_+$ is involutive.
\end{proof}
\section{Lifting Nijenhuis operators}
An interesting class of projectable Nijenhuis operators we get from the procedure of the \emph{tangent lift}. In finite dimensions, it is well known that tensor fields on a manifold $M$ can be canonically lifted to the tangent bundle $\sT M$ \cite{Grabowski:1995,Ishihara:1973, Yano:1973}. This procedure also works for Lie algebroids, replacing the canonical $\sT M$ \cite{Grabowski:1997,Grabowski:1997a}.

In infinite dimensions, tensor fields are generally not linear combinations of simple ones, so the situation is more complicated. However, for vector fields and Nijenhuis operators, the method described in \cite{Grabowski:1995} works perfectly. It is based on the \emph{canonical flip}
$$\zk:\sT\sT M\to\sT\sT M,$$
which is an isomorphism of the two canonical vector bundle structures on $\sT\sT M$. To define  $\zk$, consider a `homotopy'
$$\phi:\R^2\to M,\quad \phi(0,0)=x_0.$$
For each fixed $s\in\R$, the curve $t\mapsto \phi(t,s)$ represents a tangent vector
$$\phi'_t(s)\in\sT_{\phi(0,s)}M,\quad \phi'_t(s)=\frac{\pa}{\pa t}\,\Big|_{t=0}\phi(t,s),$$
thus a curve
$$\phi'_t:\R\to\sT M.$$
The tangent vector $\phi''_{s,t}$ to this curve at $s=0$ is an element of
$\sT_{\phi'_t(0)}\sT M$. Intertwining the order of differentiation, i.e., considering $\phi(s,t)$ instead of $\phi(t,s)$, we get an element $\phi''_{t,s}$ of $\sT_{\phi'_s(0)}\sT M$, and the association
$$\phi''_{s,t}\mapsto\phi''_{t,s}$$
defines the diffeomorphism $\zk$. Of course, $\zk^2=\id$.

In a local trivialization, we can replace $M$ with the model Banach space $E$, and $\sT\sT M$ with
$E\ti E\ti E\ti E$, whose elements we will denote
$$\bl x,\dot x,\zd x,\zd\dot x\br,$$
and view as our `local coordinates'.
With these identifications, it is easy to see that $\zk$ is really a flip,
$$\zk\bl x,\dot x,\zd x,\zd\dot x\br=\bl x,\zd x,\dot x,\zd\dot x\br.$$

\mn Still working locally, we can view a vector field as a map
$$X:E\to \sT E=E\ti E,\quad X(x)=(x,X_x).$$
Its tangent prolongation, $\sT X: \sT E\to\sT\sT E$, reads
$$\sT X(x,\dot x)=\bl x,X_x,\dot x,D_xX\br.$$
After composing with $\zk$, we get the map
$$\dt(X)=\zk\circ\sT X:\sT E\to\sT\sT E,$$
which, in fact, is a vector field on $\sT E$,
$$\dt(X)(x,\dot x)=\bl x,\dot x,X_x,D_xX\br.$$
All this extends to manifolds, and the vector field $\dt(X)$ is called the \emph{tangent lift} (or the \emph{complete lift}) of the vector field $X$.

Integral curves of $\dt(X)$ are tangent prolongations of integral curves of $X$, so $\dt(X)$ is nothing but the generator of the (local) flow $\sT\zf^t_X$, where $\zf^t_X$ is the (local) flow of $X$. From that, or by a direct inspection in `local coordinates', one can prove that the tangent lift respects the Lie bracket of vector fields  (cf. \cite{Grabowski:1995}),
\be\label{dt} \big[\dt(X),\dt(Y)\big]=\dt\big([X,Y]\big).
\ee
The tangent lift can also be defined for an Nijenhuis operator $N:\sT M\to\sT M$. Again, using the tangent functor, we get $\sT N:\sT\sT M\to\sT\sT M$. This is not a Nijenhuis operator on $\sT N$, but $\dt(N)=\zk\circ\sT N\circ\zk$ is. In our `local coordinates',
$$\dt(N)\bl x,\dot x,\zd x,\zd\dot x\br=\bl x,\dot x,N_x(\dot x),D_xN(\zd x,\dot x)+N_x(\zd\dot x)\br.$$
If $X$ is a vector field on $M$, by a straightforward check, we get
\be\label{dtN} \dt(N)\bl\dt(X)\br=\dt(NX).\ee
Combining (\ref{dt}) and (\ref{dtN}), we get
$$\cT_{\dt(N)}\bl\dt(X),\dt(Y)\br=\dt\bl\cT_N(X,Y)\br.$$
For any $v\in \sT M$, the vectors of the form $\dt(X)_v$ fill the entire tangent space $\sT_v\sT M$. Moreover, as easily seen, the fibration $\sT\zt_M:\sT \sT M\to\sT M$ projects the Nijenhuis operator $\dt(N)$ onto $N$. This proves the following.
\begin{theorem}
If $N:\sT M\to\sT M$ is a Nijenhuis operator on a Banach manifold $M$, then its tangent lift $\dt(N)$ is a Nijenhuis operator on $\sT M$. Moreover, $\dt(N)$ is projectable with respect to fibration $\sT\zt_M:\sT\sT M\to\sT M$ onto the operator $N$,
$$\xymatrix@C+30pt@R+20pt{
\sT\sT M\ar[r]^{\dt(N)}\ar[d]^{\sT\zt_M}&\sT\sT M\ar[d]^{\sT\zt_M}\\
\sT M\ar[r]^{N}&\sT M\,.}
$$
\end{theorem}

\section{The case of homogeneous Banach manifolds}
\no In \cite{Golinski:2025}, the authors consider the case in which $M=G$ is a real Banach-Lie group and $\zt:G\to G/K$ is the fibration associated with an embedded Banach subgroup $K$ of $G$. They are interested in those projectable $\cN$-operators $N$ on $G$ which are $G$-invariant, thus induce $G$-invariant $\cN$-operators on the homogeneous space $M_0=G/K$. The group $G$ acts transitively on $G/K$ by $g'.\zt(g)=\zt(g'g)$. The fibers of $\zt$ are of the form $gK$, where $g\in G$, and $K$ acts transitively on these fibers by the right translations, $g\mapsto gk$.

Let $\g$ and $\k$ be the Lie algebras of $G$ and $K$, interpreted as $\sT_eG$ and $\sT_eK\subset\sT_eG$, respectively. To simplify the notation, we will write $gX$ and $Xg$ for $\sT_eL_g(X)$ and $\sT_eR_g(X)$, where $X\in\g$, and $L_g,R_g:G\to G$ are the left/right translations by $g\in G$. The vertical bundle $V$ of the fibration $\zt:G\to G/K$ is, clearly, $V_g=g\k$. Indeed, if $X$ in $\k$, then the curve
$$t\mapsto\zt(g\exp(tX))=\zt(g)$$
is constant, so the tangent vector is zero. In particular, $\zp=\sT_e\zt:\g\to\g/\k$ is the canonical projection.

\mn An $\cN$-operator on $G$ is \emph{invariant} if
\be\label{G}N_g(Y_g)=gN_e(g^{-1}Y_g),\ee
for any $g\in G$ and any $Y_g\in\sT_gG$. In other words, $N$ maps left-invariant vector fields into left-invariant vector fields. Consequently, it is sufficient to calculate the Nijenhuis torsion for left-invariant vector fields, and $N$ is Nijenhuis if and only if $N_e$ is a Nijenhuis operator on the Lie algebra $\g$.

It is also easy to see that an invariant $\cN$-operator $N$ is projectable if and only if it is projectable along the fiber over $[e]$, i.e., along $K\subset G$.
This, in turn, is equivalent to the fact that $N_e(\k)\subset \k$ and, for every $k\in K$, $X\in\g$,
\be\label{K}N_k(Xk)-N_e(X)k\in k\k.\ee
Indeed, the curves $t\mapsto \exp(tX)$ and $t\mapsto \exp(tX)k$ project onto the same curve on $G/K$, so $\sT\zt(X)=\sT\zt(Xk)$, which implies that $X\in\g$ and $Y_k\in\sT_kG$ project onto the same vector if and only if $$(Y_k)k^{-1}-X\in\k.$$
Condition (\ref{K}) means that $N$ is also invariant with respect to the right action of $K$. Combining (\ref{K}) with (\ref{G}), we can rewrite the condition for invariant $N$ to be projectable, as (cf. \cite[Definition 2.9]{Golinski:2025})
$$N_e(\k)\subset\k\quad\text{and}\quad N_e\bl\Ad_k(X)\br-\Ad_k\bl N_e(X)\br\in\k,$$
for all $k\in K$ and all $X\in\g$. Now, Theorem 3.6 in \cite{Golinski:2025} is nothing but a particular case of our Theorem \ref{main}.

\mn If complex structures are concerned in the homogeneous case, we can work with left-invariant vector fields and reduce everything to the Lie algebra $\sT_eG=\g$ and the Lie algebra $\cN$-operator $N_e$. In this way, we recover \cite[Corollary 4.7]{Golinski:2025}.
\begin{theorem} A homogeneous $\cN$-operator $N$ on a Banach Lie group $G$ projects onto a complex structure $N_0$ on the homogeneous manifold $G/K$ if and only if $N_e^2+\id$ and $\cT_{N_e}$ take values in the Lie subalgebra $\k$. The latter is equivalent to the fact that one (thus both) of the vector spaces
$$\cZ_\pm=\{X+iY\in\g^\C M\,\big|\, Y\mp N_eX\in \k\}$$
is a Lie subalgebra in the complex Lie algebra $\g^\C$.
\end{theorem}


\vskip.5cm
\noindent Katarzyna Grabowska\\\emph{Faculty of Physics,
University of Warsaw,}\\
{\small ul. Pasteura 5, 02-093 Warszawa, Poland} \\{\tt konieczn@fuw.edu.pl}\\
https://orcid.org/0000-0003-2805-1849\\

\noindent Janusz Grabowski\\\emph{Institute of Mathematics, Polish Academy of Sciences}\\{\small ul. \'Sniadeckich 8, 00-656 Warszawa,
Poland}\\{\tt jagrab@impan.pl}\\  https://orcid.org/0000-0001-8715-2370
\\
\end{document}